\documentclass[12pt,oneside]{amsart}

\usepackage[a4paper]{geometry}
\usepackage[english]{babel}

\usepackage{lmodern}
\usepackage[all]{xy}
\usepackage{amsfonts,amssymb}
\usepackage{amsmath}
\usepackage{enumerate}
\usepackage{bbm}

\usepackage{graphicx}
\usepackage{tikz}
\usetikzlibrary{arrows}
\usepackage{pgfplots}

\usepackage[numbers,sort]{natbib}

\newcommand{\R}{\ensuremath{\mathbb{R}}}
\newcommand{\cX}{\ensuremath{\mathcal{X}}}
\newcommand{\EE}{\ensuremath{\mathbb{E}}}
\newcommand{\N}{\ensuremath{\mathbb{N}}}
\newcommand{\cD}{\ensuremath{\mathcal{D}}}
\newcommand{\cR}{\ensuremath{\mathcal{R}}}
\newcommand{\PP}{\ensuremath{\mathbb{P}}}

\newcommand{\cB}{\ensuremath{\mathcal{B}}}
\newcommand{\fF}{\ensuremath{\mathfrak{F}}}
\newcommand{\fU}{\ensuremath{\mathfrak{U}}}

\newcommand{\w}{\ensuremath{\mathbf{w}}}
\newcommand{\cH}{\ensuremath{\mathcal{H}}}
\newcommand{\h}{\ensuremath{\mathbf{h}}}
\newcommand{\wm}{{w_{\textsc{min}}}}
\newcommand{\cG}{\ensuremath{\mathcal{G}}}

\newcommand{\fW}{\ensuremath{\mathfrak{W}}}
\newcommand{\cI}{\ensuremath{\mathcal{I}}}
\newcommand{\cO}{\ensuremath{\mathcal{O}}}
\newcommand{\Zh}{\overline{\ensuremath{\mathfrak{W}}}}

\def\ind{{\mathbbm{1}}}

\usepackage{color}

\newtheorem{lemma}{Lemma}[section]
\newtheorem{theorem}[lemma]{Theorem}
\newtheorem{proposition}[lemma]{Proposition}

\newtheorem{remark}[lemma]{Remark}

\begin{document}

\title[Cutoff for non-backtracking random walks]{Cutoff for non-backtracking random walks on sparse random graphs}
\author{Anna Ben-Hamou and Justin Salez}
\address{}
\email{}

\subjclass[2010]{}

\begin{abstract}
A finite ergodic Markov chain is said to exhibit cutoff if its distance to stationarity remains close to 1 over a certain number of iterations and then abruptly drops to near 0 on a much shorter time scale. Discovered in the context of card shuffling (Aldous-Diaconis, 1986), this phenomenon is now believed to be rather typical among fast mixing Markov chains. Yet, establishing it rigorously often requires a challengingly detailed understanding of the underlying chain. Here we consider non-backtracking random walks on random graphs with a given degree sequence. Under a general sparsity condition, we establish the cutoff phenomenon, determine its precise window, and prove that the (suitably rescaled) cutoff profile approaches a remarkably simple, universal shape. 
\end{abstract}

\maketitle

\section{Introduction}

\begin{figure}[h!]
\label{fig:cutoff}
\begin{center}
{\includegraphics[angle =0,width = 8cm]{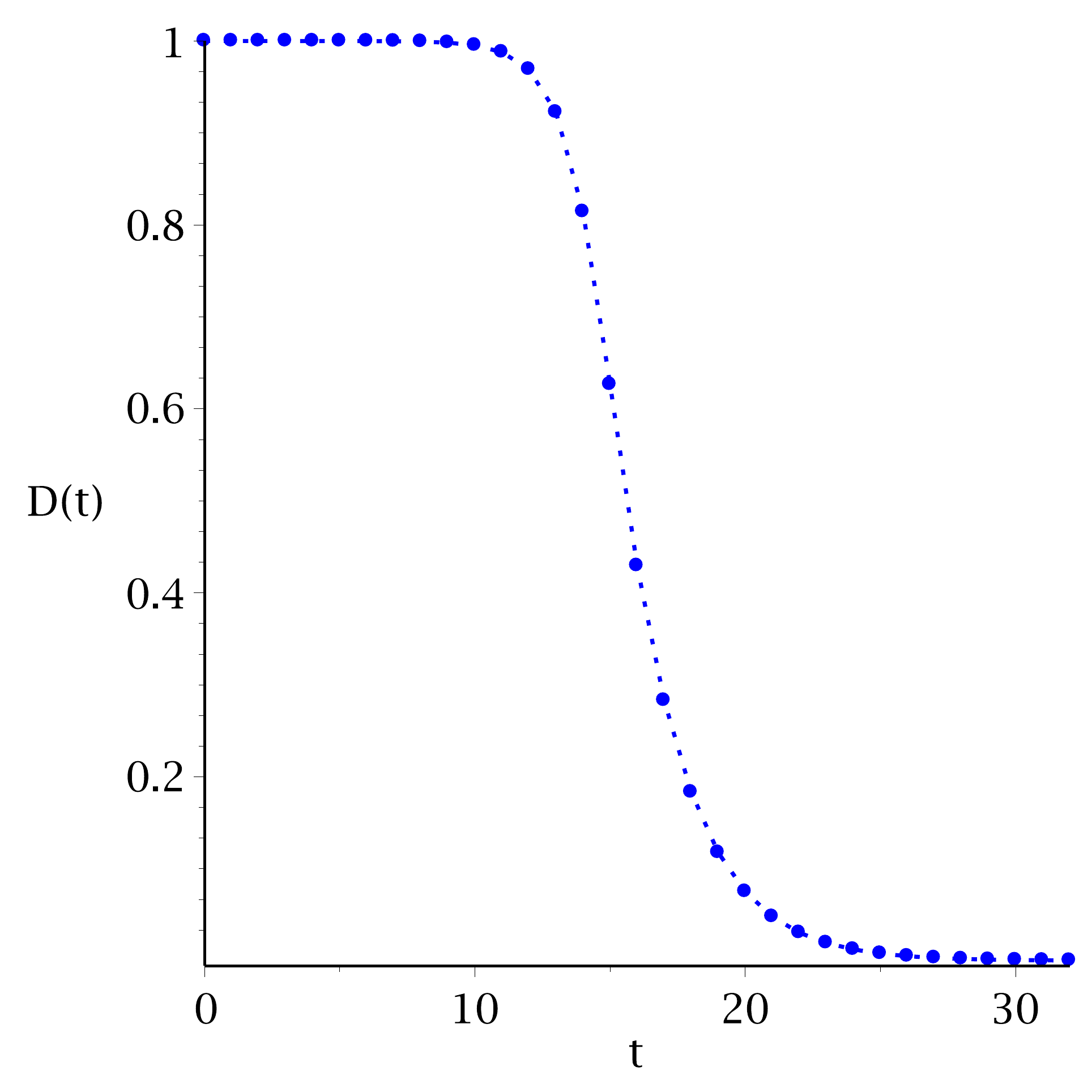}}
\caption{Distance to stationarity along time for the \textsc{nbrw} on a random graph with $10^6$ degree $3-$vertices  and $10^6$ degree $4-$vertices}
\end{center}
\end{figure}

\subsection{Setting.} Given a finite set $V$ and a function $\deg\colon V\to\{2,3,\ldots\}$  such that 
\begin{eqnarray}
\label{eq:N}
N & := & \sum_{v\in V}\deg(v)
\end{eqnarray}
is even, we construct a graph $G$ with vertex set $V$ and degrees $(\deg(v))_{v\in V}$  as follows. We  form  a set $\cX$ by  ``attaching" $\deg(v)$  \emph{half-edges} to each vertex $v\in V$: 
$$\cX:=\{(v,i)\colon v\in V,1\leq i\leq \deg(v)\}.$$
We then simply choose a \emph{pairing} $\pi$ on $\cX$ (i.e., an involution without fixed points), and interpret every pair of matched half-edges $\{x,\pi(x)\}$ as an edge between the corresponding vertices. Loops and multiple edges are allowed. 
\begin{figure}[h!]
\begin{center}
\boxed{\includegraphics[angle =0,width = 4.3cm]{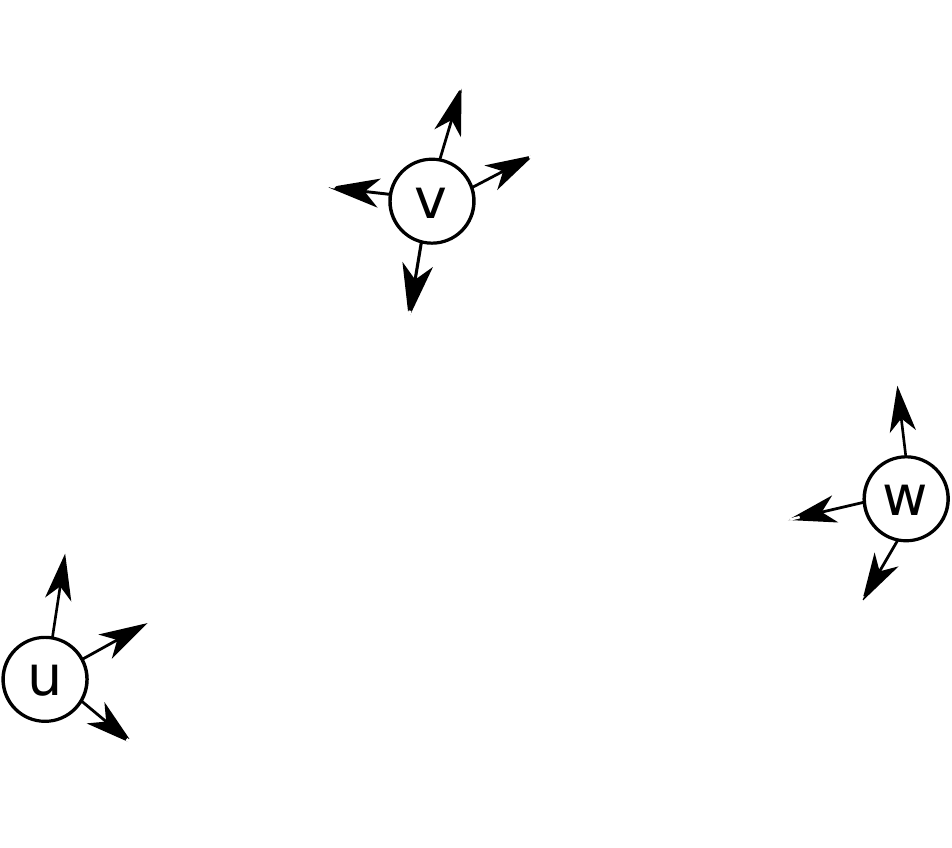}}
\!\!\boxed{\includegraphics[angle =0,width = 4.3cm]{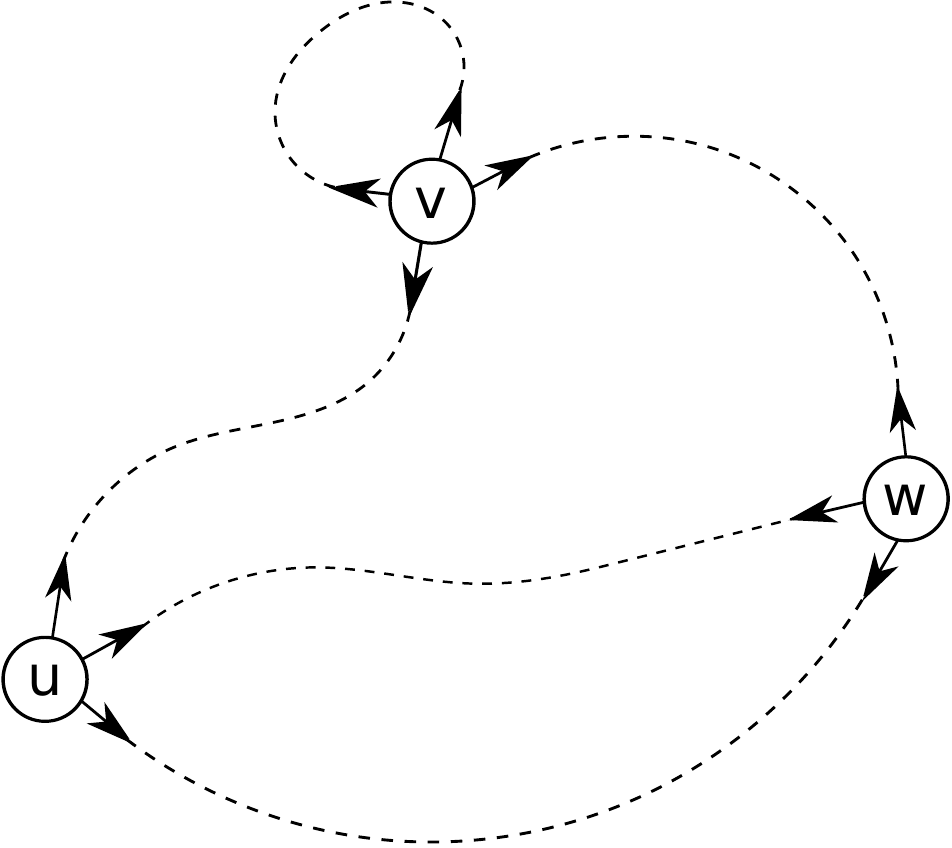}}
\!\!\boxed{\includegraphics[angle =0,width = 4.3cm]{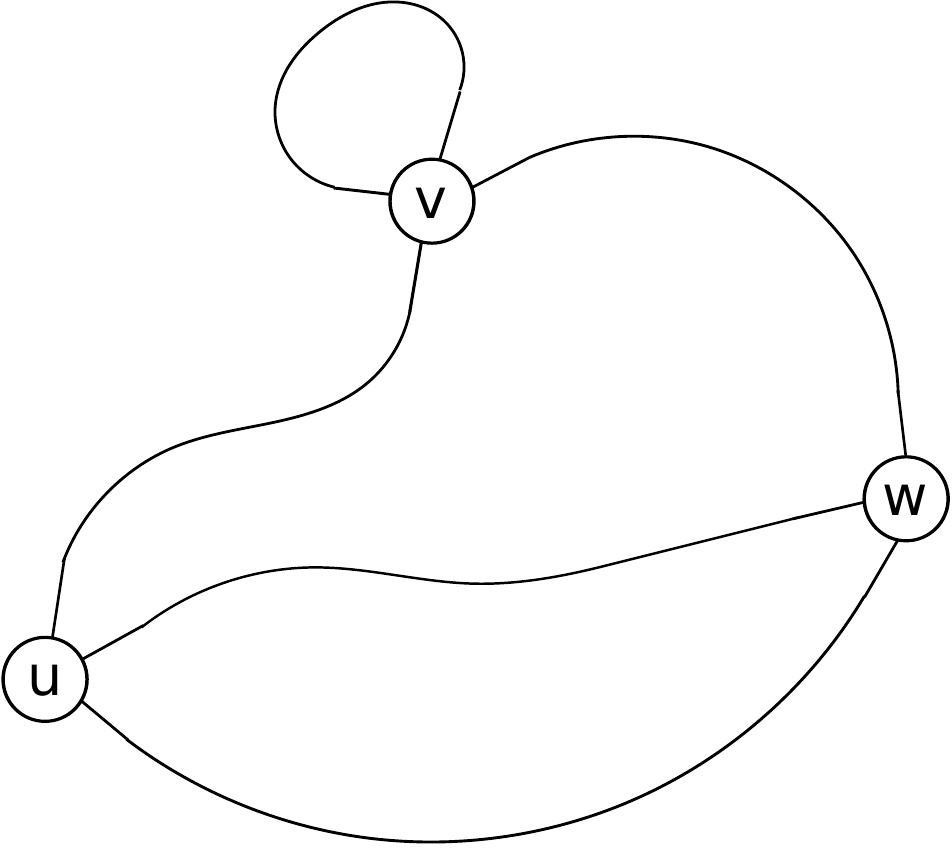}}
\caption{A set of half-edges $\cX$, a pairing $\pi$ and the resulting graph $G$}
\end{center}
\end{figure}

The \emph{non-backtracking random walk} (\textsc{nbrw}) on the graph $G=G(\pi)$ is a discrete-time Markov chain with state space $\cX$ and transition matrix
$$P(x,y)=
\left\{
\begin{array}{ll}
\frac{1}{\deg(\pi(x))} & \textrm{ if y is a neighbour of } \pi(x) \\
0 & \textrm{ otherwise}.
\end{array}
\right.
$$
In this definition and throughout the paper, two half-edges $x=(u,i)$ and $y=(v,j)$ are called \emph{neighbours} if $u=v$ and $i\neq j$, and we let $\deg(x):=\deg(u)-1$ denote the number of neighbours of the half-edge $x=(u,i)$. In words, the chain moves at every step from the current state $x$ to a uniformly chosen neighbour of $\pi(x)$. 
\begin{figure}[h!]
\begin{center}
{\includegraphics[angle =0,width = 4.5cm]{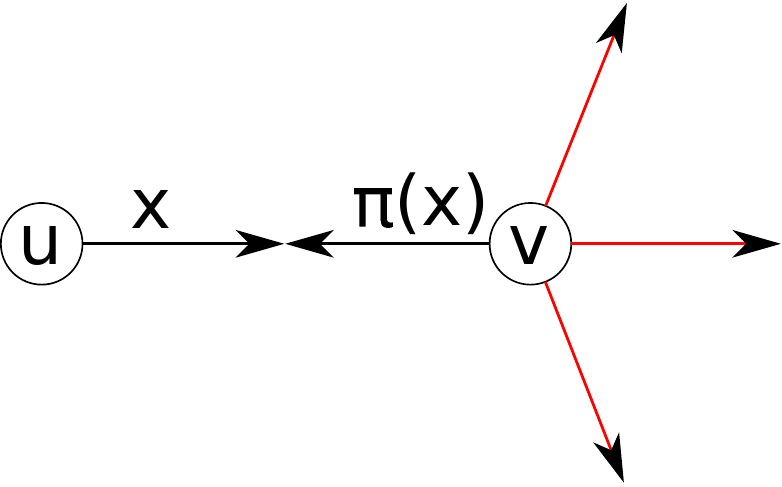}}
\caption{The non-backtracking moves from $x$ (in red)}
\end{center}
\end{figure}

Note that the matrix $P$ is symmetric with respect to $\pi$: for all $x,y\in\cX$, 
  \begin{eqnarray}
\label{sym}
P(\pi(y),\pi(x))=P(x,y).
\end{eqnarray}
In particular, $P$ is doubly stochastic: the uniform law on $\cX$ is invariant for the chain. The \emph{worst-case total-variation distance to equilibrium} at time $t\in\N$ is 
  \begin{eqnarray}
\label{cd}
\cD(t)  & := & \max_{x\in\cX}\left\{\frac{1}{2}\sum_{y\in\cX}\left|P^t(x,y)-\frac{1}{N}\right|\right\}.
\end{eqnarray}
$\cD$ is non-increasing, and the number of transitions that have to be made before it falls below a given threshold $0<\varepsilon<1$ is known as the \emph{mixing time}:
\begin{eqnarray*}
t_{\textsc{mix}}(\varepsilon) & := & \inf\left\{t\in\N\colon \cD(t)<\varepsilon\right\}.
\end{eqnarray*}

\subsection{Result.} The present paper is concerned with the typical profile of the function $t\mapsto \cD(t)$ under the so-called  \emph{configuration model}, i.e. when the pairing $\pi$ is chosen uniformly at random among the $(N-1)!!$ possible pairings of $\cX$. 
In order to study  large-size asymptotics, we let the vertex set $V$ and degree function $\deg\colon V\to\N$ depend on an implicit parameter $n\in\N$, which we omit from the notation for convenience. The same convention applies to all related quantities, such as $N$ or $\cX$. All asymptotic statements are understood as $n\to\infty$. Our interest is in the \emph{sparse regime}, where the number $N$ of half-edges diverges at a much faster rate than the maximum degree. Specifically, we assume that
\begin{eqnarray}
\label{assume:Delta}
\Delta :=\max_{v\in V}\deg(v) & = & N^{o(1)}.
\end{eqnarray}
As the behaviour of the \textsc{nbrw} at degree-2 vertices is deterministic, we will also assume, without much loss of generality, that
\begin{eqnarray}
\label{assume:delta}
\min_{v\in V}\deg(v)&\geq& 3.
\end{eqnarray}
Remarkably enough, the asymptotics in this regime depend on the degrees through two simple statistics: the mean logarithmic degree of an half-edge
\begin{eqnarray}\label{eq:mu}
\mu &:=& \frac{1}{N}\sum_{v\in V}\deg(v)\log\left(\deg(v)-1\right),  
\end{eqnarray}
and the corresponding variance
\begin{eqnarray}\label{eq:sigma}
\sigma^2 & := & \frac{1}{N}\sum_{v\in V}\deg(v)\left\{\log\left(\deg(v)-1\right)-\mu\right\}^2.
\end{eqnarray}
We will also need some control on the third absolute moment:
\begin{eqnarray}\label{eq:varrho}
\varrho & := & \frac{1}{N}\sum_{v\in V}\deg(v)\left|\log\left(\deg(v)-1\right)-\mu\right|^3.
\end{eqnarray}
It might help the reader to think of $\mu,\sigma$ and $\varrho$ as being fixed, or bounded away from $0$ and $\infty$. However, we only impose the following (much weaker) condition:
\begin{eqnarray}
\label{assume:sigma}
{\frac{\sigma^2}{\mu^{3}}} >> \frac{(\log\log N)^2}{{\log N}} & \textrm{ and } & \frac{\sigma^3}{\varrho\sqrt{\mu}}>>\frac 1{\sqrt{{\log N}}}.
\end{eqnarray}
Our main result states that on most graphs with degrees $(\deg(v))_{v\in V}$, the \textsc{nbrw} exhibits a remarkable behaviour, visible on Figure \ref{fig:cutoff} and known as a \emph{cutoff}: the distance to equilibrium remains close to $1$ for a rather long time, roughly  
\begin{eqnarray}
\label{eq:t_star}
t_\star:=\frac{\log N}{\mu},
\end{eqnarray}
and then abruptly drops to nearly $0$ over a much shorter time scale\footnote{The fact that $\omega_\star<<t_\star$ follows from condition (\ref{assume:Delta}).}, of order
\begin{eqnarray}
\label{eq:omega_star}
\omega_\star:=\sqrt{\frac{{\sigma^2\log N}}{\mu^{3}}}.
\end{eqnarray}
Moreover, the cutoff shape inside this window approaches  a surprisingly simple function $\Phi\colon\R\to[0,1]$, namely the tail distribution of the standard normal:
$$\Phi(\lambda):=\frac{1}{2\pi}\int_{\lambda}^\infty e^{-\frac {u^2}2}\,{\rm d} u.$$
It is remarkable that this limit shape does not depend at all on the precise degrees. 
\begin{theorem}[Cutoff for the \textsc{nbrw} on sparse graphs]
\label{thm:main}
For every $0<\varepsilon<1$, 
\begin{eqnarray*}
\frac{t_{\textsc{mix}}(\varepsilon)-t_\star}{\omega_\star} & \xrightarrow[]{\PP} & {\Phi}^{-1}(\varepsilon).
\end{eqnarray*}
Equivalently, for $t=t_\star+\lambda\omega_\star+o(w_\star)$ with  $\lambda\in\R$ fixed, we have 
$\cD(t)  \xrightarrow[]{\PP}  \Phi(\lambda).$
\end{theorem}

\subsection{Comments}
It is interesting to compare this with the $d-$regular case (i.e., $\deg\colon V\to\N$  constant equal to $d$) studied by \citet{lubetzky2010cutoff}: by a remarkably precise path counting argument, they establish cutoff within constantly many steps  around $t_\star=\log N/\log(d-1)$. 
To appreciate the effect of heterogeneous degrees, recall that  $\mu$ and $\sigma$ are the mean and variance of $\log D$, where $D$ is the degree of a uniformly sampled half-edge. Now, by Jensen's Inequality,
\begin{eqnarray*}
t_\star & \geq & \frac{\log N}{\log\EE[D]}\, ,
\end{eqnarray*}
and the less concentrated $D$, the larger the gap. The right-hand side is a well-known characteristic length in $G$, namely the typical inter-point distance (see e.g., \cite{remco}). One notable effect of heterogeneous degrees is thus that the mixing time becomes significantly larger than the natural graph distance. A heuristic explanation is as follows: in the regular case, all paths of length $t$ between two points are equally likely for the \textsc{nbrw}, and mixing occurs as soon as $t$ is large enough for many such paths to exist.  In the non-regular case however, different paths have very different weights, and most of them actually have a negligible chance of being  seen by the walk. Consequently, one has to make $t$ larger in order to see paths with a ``reasonable'' weight. Even more remarkable is the impact of heterogeneous degrees on the cutoff width $\omega_\star$, which satisfies $\omega_\star>>\log\log N$ against $\omega_\star=\Theta(1)$ in the regular case. Finally, the gaussian limit shape $\Phi$ itself is specific to the non-regular case and is directly related to the fluctuations of degrees along a typical trajectory of the \textsc{nbrw}.

\begin{remark}[Simple graphs]A classical result by \citet{janson2009probability} asserts that the graph produced by the configuration model is simple (no loops or multiple edges) with  probability asymptotically bounded away from $0$, as long as
\begin{equation}\label{eq:simple-graph}
\sum_{v\in V}\deg(v)^2=\cO(N)\, .
\end{equation}
Moreover, conditionally on being simple, it is uniformly distributed over all simple graphs with degrees $(\deg(v))_{v\in V}$. Thus, every property which holds {\bf whp} under the configuration model also holds {\bf whp} for the uniform simple graph model. In particular, under (\ref{eq:simple-graph}), the conclusion of Theorem \ref{thm:main} extends to simple graphs.
\end{remark}
\begin{remark}[\textsc{IID} degrees]
A common setting consists in generating an infinite \textsc{iid} degree sequence $(\deg(v))_{v\in\N}$ from some fixed degree distribution $Q$ and then restricting it to the index set $V=\{1,\ldots,n\}$ for each $n\geq 1$.  Let $D$ denote a random integer with distribution $Q$. Assuming that
$$
\begin{array}{ccc}
\PP\left(D\leq 2\right)=0, & 
\mathrm{Var}(D)>0, \textrm{ and }& \EE\left[e^{\theta D}\right]<\infty\textrm{ for some }\theta>0,
\end{array}
$$
ensures that the conditions (\ref{assume:Delta}), (\ref{assume:delta}) and (\ref{assume:sigma}) hold almost surely. Thus, Theorem \ref{thm:main} applies with the parameters $\mu,\sigma$ and $N$ now being random. But the latter clearly concentrate around their deterministic counterparts, in the following sense:
\begin{eqnarray*}
N&=& n\EE[D]+O_\PP\left(n^{\frac 12}\right)\\
\mu&=& \mu_\star+O_\PP\left(n^{-\frac 12}\right)\ \textrm{ with }\ \mu_\star={\EE[D\log(D-1)]}/{\EE[D]}\\
\sigma&=&\sigma_\star+O_\PP\left(n^{-\frac 12}\right)\ \textrm{ with }\ \sigma_\star^2={\EE\left[D\left\{\log(D-1)-\mu_\star\right\}^2\right]}/{\EE[D]}. 
\end{eqnarray*}
Those error terms are small enough to allow one to substitute $n,\mu_\star,\sigma_\star$ for $N,\mu,\sigma$ without affecting the convergence stated in Theorem \ref{thm:main}.
\end{remark}

\subsection{Related work.}

The first instances of the cutoff phenomenon were discovered in the early 80's by \citet{diaconis1981generating} and  \citet{aldous1983mixing}, in the context of card shuffling: given a certain procedure for shuffling a deck of cards, there exists a quite precise number of shuffles slightly below which the deck is far from being mixed, and slightly above which it is almost completely mixed. The term \emph{cutoff} and the general formalization appeared shortly after, in the seminal paper by \citet{aldous1986shuffling}. Since then, this remarkable behaviour has been identified in a variety of other contexts, see e.g., \citet{diaconis1996cutoff}, \citet{chen2008cutoff}, or the survey by \citet{saloff2004random} for random walks on finite groups. 

Interacting particle systems in statistical mechanics provide a rich class of dynamics displaying cutoff. One emblematic example is the \emph{stochastic Ising model} at high enough temperature, for which the cutoff phenomenon  has been established successively on the complete graph (\citet{levin2010glauber}), on lattices (\citet{ding2009mixing}, \citet{lubetzky2014cutoff}), and finally on any sequence of graphs (\citet{lubetzky2014universality}). Other examples include the \emph{Potts model} (\citet{cuff2012glauber}), the \emph{East process} (\citet{ganguly2013cutoff}), or the \emph{Simple Exclusion process} on the cycle (\citet{lacoin2015cutoff}).

The problem of singling out abstract conditions under which the cutoff phenomenon occurs, without necessarily pinpointing its precise location, has drawn considerable attention. 
In 2004, \citet{peresamerican} proposed a simple spectral criterion for reversible chains, known as the \emph{product condition}. Although counter-examples have quickly been constructed (see \citet[Chapter 18]{levin2009markov} and \citet[Section 6]{chen2008cutoff}), the condition is widely believed to be sufficient for ``most" chains. This has already been verified for birth-and-death chains (\citet{ding2010total}) and, more generally, for random walks on trees (\citet{basucharacterization}). The latter result relies on a promising characterization of cutoff in terms of the concentration of hitting times of ``worst'' (in some sense) sets. See also \citet{oliveira2012mixing}, \citet{sousi2013hitting}, \citet{griffiths2014tight} and \citet{hermon2015technical}.

Many natural families of Markov chains are now believed to exhibit cutoff. Yet, establishing this phenomenon rigorously requires a very detailed understanding of the underlying chain, and often constitutes a challenging task even in situations with a high degree of symmetry. The historical case of random walks on the symmetric group for example, is still far from being completely understood: see \citet{saloff2004random} for a list of  open problems, and \citet{berestycki2014cutoff} for a recent proof of one of them.

Understanding the mixing properties of random walks on sparse random graphs constitutes an important theoretical problem, with applications in a wide variety of contexts (see e.g., the survey by \citet{cooper2011random}). A classical result of \citet{broder1987expander}   states that random $d-$regular graphs with $d$ fixed are \emph{expanders} with high probability (see also  \citet{friedman2008proof}). In particular, the simple random walk (\textsc{srw}) on such graphs satisfies the product condition, and should therefore exhibit cutoff. This  long-standing conjecture was confirmed only  recently in an impressive work by \citet{lubetzky2010cutoff}, who also determined the precise cutoff window and profile. Their result is actually derived from the analysis of the \textsc{nbrw} itself, via a clever transfer argument.  Interestingly, the mixing time of the \textsc{srw} is $d/(d-2)$ times larger than that of the \textsc{nbrw}. This confirms the practical advantage of \textsc{nbrw} over \textsc{srw} for efficient network sampling and exploration, and complements a well-known spectral comparison for regular expanders due to \citet{alon2007non}, as well as a recent result by \citet{cooper2014vacant} on the cover time of random regular graphs. For other ways of speeding up random walks, see \citet{cooper2011random}.  

In the non-regular case however, the tight correspondence between the \textsc{srw} and the \textsc{nbrw} breaks down, and there seems to be no direct way of transferring our main result to the \textsc{srw}. We note that the latter should exhibit cutoff since the product condition holds, as can be seen from the fact that the \emph{conductance} of sparse random graphs with a given degree sequence remains bounded away from $0$ (see \citet{abdullah2012cover}).  Confirming this constitutes a challenging open problem. In particular, it would be interesting to see whether the \textsc{srw} still mixes faster than the \textsc{nbrw}. To the best of our knowledge, no precise conjectural expression for the mixing time of the \textsc{srw} has been put forward\footnote{During the finalization of the manuscript, a solution to this problem has been announced \cite{2015arXiv}.}.

\section{Proof outline}
\label{sec:outline}
The proof of Theorem \ref{thm:main} is divided into two (unequal) halves:  for
\begin{eqnarray}
\label{def:t}
t & = & t_\star+\lambda w_\star+o(w_\star),
\end{eqnarray}
with $\lambda\in\R$ fixed, we will show that
\begin{eqnarray}
\label{lowerbound}\EE\left[\cD\left(t\right)\right]& \geq & \Phi(\lambda) - o(1)\, ,\\
\label{upperbound}\cD\left(t\right) & \leq & {\Phi}(\lambda)+o_\PP(1).
\end{eqnarray}
The lower bound (\ref{lowerbound}) is proved in Section \ref{sec:lower-bound}. The difficult part is the upper bound (\ref{upperbound}), due to the maximization over all possible initial states.  Starting from state $x\in\cX$, the distance to stationarity is
\begin{eqnarray*}
\cD_x(t) & = & \sum_{y\in\cX} \left(\frac{1}{N}-P^t(x,y)\right)_+\\
& = &\sum_{y\in\cX} \left(\frac{1}{N}-P^t(x,\pi(y))\right)_+,
\end{eqnarray*}
since $\pi$ is an involution. Using the symmetry (\ref{sym}), we may re-write $P^t(x,\pi(y))$ as
\begin{eqnarray}
\label{eq:reversible}
P^t(x,\pi(y)) & = & \sum_{(u,v)\in \cX\times\cX} P^{t/2}(x,u)P^{t/2}(y,v) \ind_{\{\pi(u)=v\}} \, .
\end{eqnarray}
As a first approximation, let us assume that the balls of radius $t/2$ around $x$ and $y$ consist of disjoint trees, in agreement with Figure \ref{fig:tree}. 
\begin{figure}[h!]
\label{fig:tree}
\begin{center}
{\includegraphics[angle =0,width = 8cm]{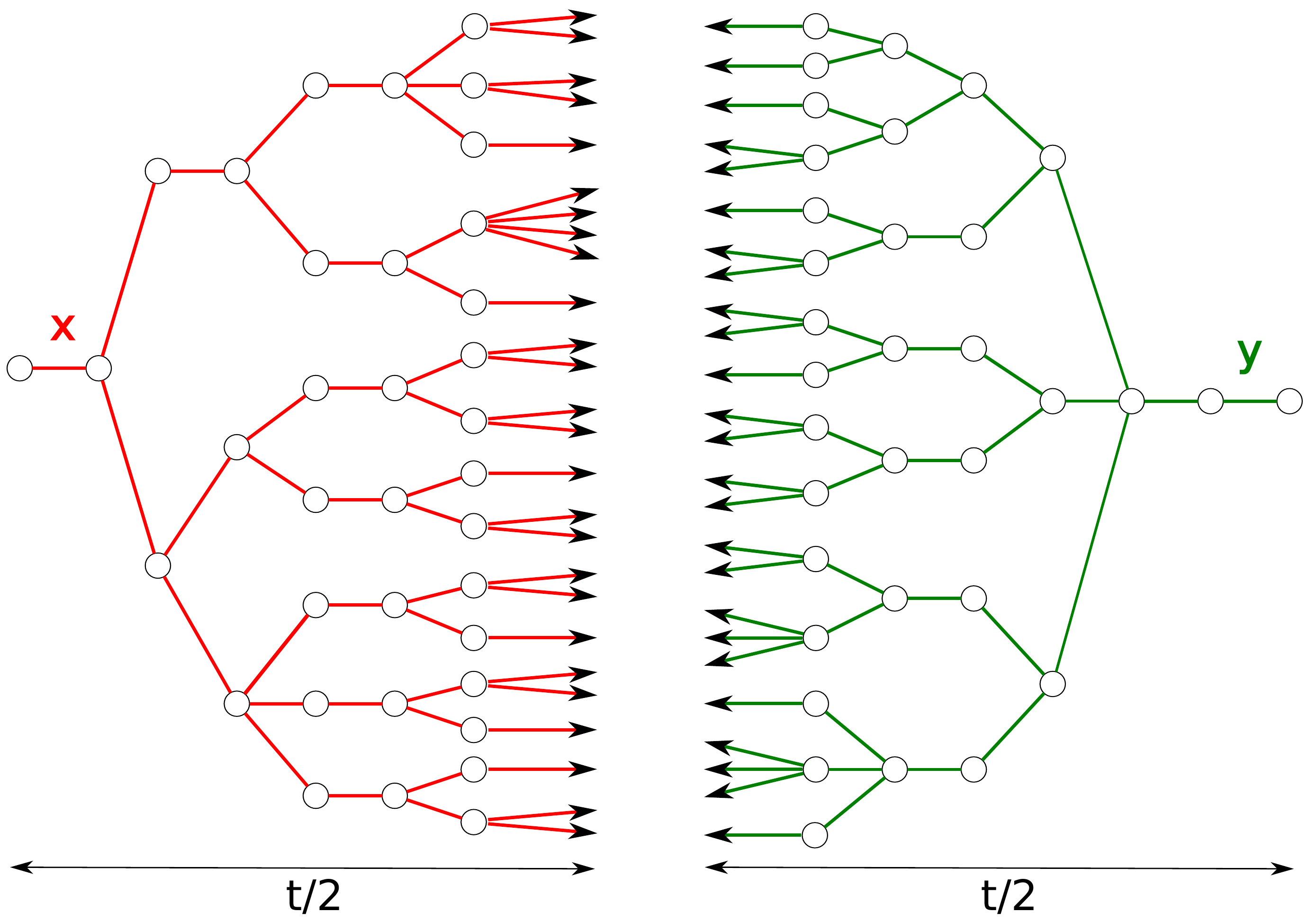}}
\caption{The tree-approximation}
\end{center}
\end{figure}
This is made rigorous by a particular exposure process described in Section \ref{sec:exposure-process}. The \emph{weight} $\w(u):=P^{t/2}(x,u)$ (resp. $\w(v):=P^{t/2}(y,v)$) can then be unambiguously written as the inverse product of degrees along the unique path from $x$ to $u$ (resp. $y$ to $v$).

A second approximation consists in eliminating those paths whose weight exceeds some given threshold $\theta > 0$ (the correct choice turns out to be $\theta\approx \frac 1N$):
\begin{displaymath}
P^t(x,\pi(y))\approx \sum_{u,v}\w(u)\w(v)\ind_{\w(u)\w(v)\leq \theta} \ind_{\{\pi(u)=v\}} \, .
\end{displaymath}
Conditionally on the two trees of height $t/2$, this is a weighted sum of weakly dependent Bernoulli variables, and the large-weight truncation should prevent it from deviating largely from its expectation. We make this argument rigorous in Section \ref{sec:completion}, using Stein's method of exchangeable pairs. Provided the exposure process did not reveal too many pairs of matched half-edges, the conditional expectation of $\ind_{\{\pi(u)=v\}}$ remains close to $1/N$, and we obtain the new approximation
\begin{displaymath}
N P^t(x,\pi(y))\approx \sum_{u,v}\w(u)\w(v)\ind_{\w(u)\w(v)\leq \theta} \, .
\end{displaymath}

Now, the right-hand side corresponds to the quenched probability (conditionally on the graph) that the product of the weights seen by two independent \textsc{nbrw} of length $t/2$, one starting from $x$ and the other from $y$, does not exceed $\theta$. The last step consists in approximating those trajectories by independent uniform samples from $\cX$, which we denote by $X_1^\star,\ldots,X_t^\star$. We obtain
\begin{eqnarray*}
\sum_{u,v}\w(u)\w(v)\ind_{\w(u)\w(v)\leq \theta}&\approx& \PP\left[\frac{1}{\deg(X_1^\star)}\cdots\frac{1}{\deg(X_t^\star)}\leq \theta\right]\\
&\approx& \PP\left[\frac{\sum_{k=1}^t(\mu-\log \deg(X^\star_k))}{\sigma\sqrt{t}}\leq \frac{\mu t+\log \theta}{\sigma\sqrt{t}}\right]\\
&\approx & 1-\Phi(\lambda)\, ,
\end{eqnarray*}
by the central limit theorem, since $\theta\approx 1/N$ and $t\approx t_\star+\lambda\omega_\star$. Consequently, 
\begin{eqnarray*}
\sum_y \left(\frac{1}{N}-P^t(x,\pi(y))\right)_+&\approx& \Phi(\lambda)\, ,
\end{eqnarray*}
as desired. This argument is made rigorous in Sections \ref{sec:small}, \ref{sec:large} and \ref{sec:collisions}. 

\section{The lower bound}\label{sec:lower-bound}

Fix $t\geq 1$, two states  $x,y\in\cX$, and a parameter $\theta\in(0,1)$. Let $P^t_\theta(x,y)$ denote the contribution to $P^t(x,y)$ from paths having weight less than $\theta$. 
Note that $P^t_\theta(x,y)< P^t(x,y)$ if and only if some path of length $t$ from $x$ to $y$ has weight larger than $\theta$, implying in particular that $P^t(x,y)>\theta$. Thus,
\begin{eqnarray*}
\frac{1}{N}-P^t_\theta(x,y) & \leq & \left(\frac{1}{N}-P^t(x,y)\right)_+ 
+\frac{1}{N}{\bf 1}_{P^t(x,y)>\theta}.
\end{eqnarray*}
Summing over all $y\in\cX$ and observing that there can not be more than $1/\theta$ half-edges $y\in\cX$ satisfying $P^t(x,y)>\theta$, we obtain
\begin{eqnarray*}
1-\sum_{y\in\cX}P^t_\theta(x,y) & \leq  & \cD_x(t) + \frac{1}{\theta N}.
\end{eqnarray*}
Now, the left-hand side is the quenched probability (i.e., conditional on the underlying pairing) that a \textsc{nbrw} $\{X_k\}_{0\leq k \leq t}$ starting at $x$ satisfies $\prod_{k=1}^t\frac{1}{\deg(X_k)}> \theta$.
Taking expectation w.r.t. the pairing, we arrive at
\begin{eqnarray}
\label{lb:1}
\PP\left(\prod_{k=1}^t\frac{1}{\deg(X_k)}> \theta\right) & \leq & \EE[\cD_x(t)] + \frac{1}{\theta N},
\end{eqnarray}
where the average is now over both the \textsc{nbrw} and the pairing (annealed law). 
A useful property of the uniform pairing is that it can be constructed sequentially, the pairs being revealed along the way, as we need them. We exploit this degree of freedom to generate  the walk $\{X_k\}_{k\geq 0}$ and the pairing simultaneously, as follows. Initially, all half-edges are unpaired and $X_0=x$; then at each time $k\geq 1$, 
\begin{enumerate}
\item if $X_{k-1}$ is unpaired, we pair it with a uniformly chosen other unpaired half-edge; otherwise, $\pi(X_{k-1})$ is already defined and no new pair is formed.
\item in both cases, we let $X_{k}$ be a uniformly chosen neighbour of $\pi(X_{k-1})$.
\end{enumerate}
 The sequence $\{X_k\}_{k\geq 0}$ is then exactly distributed according to the annealed law. Now, if we sample uniformly from $\cX$ instead of restricting the random choice made at (i) to unpaired half-edges, the uniform neighbour chosen at step (ii) also has the uniform law on $\cX$. This creates a coupling between the process $\{X_k\}_{k\geq 1}$ and a sequence $\{X_k^\star\}_{k\geq 1}$ of \textsc{iid} samples from $\cX$, valid until the first time $T$ where the uniformly chosen half-edge or its uniformly chosen neighbour is already paired. As there are less than $2k$ paired half-edges by step $k$, a crude union-bound yields
$$\PP\left(T\leq t\right)\leq \frac{2t^2}{N}.$$
Consequently, 
\begin{eqnarray}
\label{lb:2}
\left|\PP\left(\prod_{k=1}^t\frac{1}{\deg(X_k)}> \theta\right) - \PP\left(\prod_{k=1}^t\frac{1}{\deg(X_k^\star)}> \theta\right)\right| & \leq & \frac{2t^2}{N}.
\end{eqnarray}
On the other hand, since $\{X_1^\star,\ldots,X_t^\star\}$ are \textsc{iid}, Berry-Esseen's inequality implies
\begin{eqnarray}
\label{lb:3}
\left|\PP\left(\prod_{k=1}^t\frac{1}{\deg(X_k^\star)}> \theta\right)-\Phi\left(\frac{\mu t+\log\theta}{\sigma\sqrt t}\right)\right| & \leq & \frac{\varrho}{\sigma^3\sqrt{t}}.
\end{eqnarray}
We may now combine (\ref{lb:1}), (\ref{lb:2}), (\ref{lb:3}) to obtain
\begin{eqnarray*}
\EE[\cD_x(t)] & \geq &  \Phi\left(\frac{\mu t+\log\theta}{\sigma\sqrt t}\right) - \frac{1}{\theta N}-\frac{2t^2}{N}- \frac{\varrho}{\sigma^3\sqrt{t}}.
\end{eqnarray*}
With $t$ as in (\ref{def:t}) and $\theta=(\log N)/N$, the right-hand side is $\Phi(\lambda)+o(1)$, thanks to our assumptions on $\mu,\sigma,\varrho$. This establishes the lower bound (\ref{lowerbound}).

\section{The upper-bound}

Following \citet{lubetzky2010cutoff}, we call $x\in\cX$ a \emph{root} (written $x\in\cR$) if the (directed) ball of radius $h$ centered at $x$ (denoted by $\cB_x$) is a tree, where  
\begin{eqnarray}
\label{def:h}
h & := & \left\lfloor\frac{\log N}{10\log \Delta}\wedge\log\log N\right\rfloor.
\end{eqnarray}
Note that $1<<h<<\omega_\star$ by assumptions (\ref{assume:Delta}) and (\ref{assume:sigma}).
The first proposition below shows that we may restrict our attention to paths  between roots. 
The second proposition provides a good control on such paths. 
\begin{proposition}[Roots are quickly reached] 
\label{pr:root} 
\begin{eqnarray*}
\max_{x\in\cX}P^h(x, \cX\setminus\cR) & \xrightarrow[]{\PP} & 0.
\end{eqnarray*}
\end{proposition}
\begin{proposition}[Roots are well inter-connected]  For $t$ as in (\ref{def:t}),
\label{pr:up}
$$\min_{x\in\cR}\min_{y\in\cR\setminus\cB_x}P^t(x,\pi(y))\geq \frac{1-\Phi(\lambda)-o_\PP(1)}{N}.$$
\end{proposition}
Let us first see how those results imply the upper-bound (\ref{upperbound}). Observe that
\begin{eqnarray*}
\cD(t+h)&\leq & \max_{x\in\cX}P^h\left(x,\cX\setminus\cR\right)+\max_{x\in \cR}\cD_x(t)\, .
\end{eqnarray*}
The first term is $o_\PP(1)$ by Proposition \ref{pr:root}. For the second one, we write
\begin{eqnarray*}
\cD_x(t) & = & \sum_{y\in \cR\setminus\cB_x}\left(\frac{1}{N}-P^t(x,\pi(y))\right)_++\sum_{y\in \cB_x\cup(\cX\setminus\cR)}\left(\frac{1}{N}-P^t(x,\pi(y))\right)_+.
\end{eqnarray*}
Proposition \ref{pr:up} ensures that the first sum is bounded by $\Phi(\lambda)+o_\PP(1)$ uniformly in $x\in\cR$. To see that the second sum is $o_\PP(1)$ uniformly in $x\in\cR$, it suffices to bound its summands by $1/N$ and observe that $|\cB_x|  \leq  \Delta^h = o( N)$ by (\ref{def:h}), while
\begin{eqnarray*}
|\cX\setminus\cR| & = & \sum_{x\in\cX}P^h(x,\cX\setminus\cR),
\end{eqnarray*}
($P$ is doubly stochastic), which is $o_\PP(N)$ uniformly in $\cX$ by Proposition \ref{pr:root}.

\begin{proof}[Proof of Proposition \ref{pr:root}]
Define $r:=\left\lfloor\frac{\log N}{5\log \Delta}\right\rfloor$ and fix $x\in\cX$. The ball of radius $r$ around $x$ can be generated sequentially, its half-edges being paired one after the other with uniformly chosen other unpaired half-edges, until the whole ball has been paired. Observe that at most $k=\frac{\Delta\left((\Delta-1)^{r}-1\right)}{\Delta-2}$ pairs are formed. Moreover, for each of them, the number of unpaired half-edges having an already paired neighbour is at most  $\Delta(\Delta-1)^r$ and hence the conditional chance of hitting such an half-edge (thereby creating a cycle) is at most $p=\frac{\Delta(\Delta-1)^r-1}{N-2k-1}$. Thus, the probability that more than one cycle is found is at most
\begin{eqnarray*}
(kp)^2 & = & O\left(\frac{(\Delta-1)^{4r}}{N^2}\right)\ = \ o\left(\frac{1}{N}\right).
\end{eqnarray*}
Summing over all $x\in\cX$ (union bound), we obtain that with high probability, no ball of radius $r$ in $G(\pi)$ contains more than one cycle. 

To conclude the proof, we now fix a pairing $\pi$ with the above property, and we prove that the \textsc{nbrw} on $G(\pi)$ starting from any $x\in\cX$ satisfies
\begin{eqnarray}
\label{toshow}
\PP\left(X_t\textrm{ is not a root}\right) & \leq & 2^{1-t},
\end{eqnarray}
for all $t\leq r-h$. The claim is trivial if the ball of radius $r$ around $x$ is acyclic. Otherwise, it contains a single cycle $\mathcal{C}$, by assumption. Write $d(z,\mathcal{C})$ for the minimum length of a non-backtracking path from $x$ to some $z\in \mathcal{C}$. 
The non-backtracking property ensures that if $d(X_{t},\mathcal{C})<d(X_{t+1},\mathcal{C})$ for some $t<r-h$, then $X_{t+1},X_{t+2},\ldots,X_{r-h}$ are all roots. By (\ref{assume:delta}), the conditional chance that  $d(X_{t+1},\mathcal{C})=d(X_t,\mathcal{C})+1$ given the past is at least $1/2$ (unless  $d(X_t,\mathcal{C})=1$, which can only happen once). This shows (\ref{toshow}). We then specialize to $t=h$.
\end{proof}

\section{The exposure process}\label{sec:exposure-process} 
The remainder of the paper is devoted to the proof of Proposition \ref{pr:up}. Fix two distinct half-edges $x,y\in\cX$. We  describe a two-stage procedure that generates a uniform pairing on $\cX$ together with a rooted forest $\fF$ keeping track of certain paths from $x$ and $y$.  Initially, all half-edges are unpaired and $\fF$ is reduced to its two roots, $x$ and $y$. We then iterate the following three steps:
\begin{enumerate}
\item[1.] An unpaired half-edge $z\in\fF$ is selected according to some rule (see below).
\item[2.] $z$ is paired with a uniformly chosen other unpaired half-edge $z'$.
\item[3.] If neither $z'$ nor any of its neighbours was already in $\fF$, then all neighbours of $z'$ become children of $z$ in the forest $\fF$. 
 \end{enumerate}
The \emph{exploration stage} stops when no unpaired half-edge is compatible with the selection rule. We then complete the pairing by matching all the remaining unpaired half-edges uniformly at random: this is the \emph{completion stage}. 

The condition in step $3$ ensures that $\fF$ remains a forest: any $z\in\fF$ determines a unique sequence $(z_0,\ldots,z_h)$ in $\fF$  such that $z_0$ is a root, $z_{i}$ is a child of $z_{i-1}$ for each $1\leq i\leq h$, and $z_h=z$. We shall naturally refer to $h$  and $z_0$ as the \emph{height} and \emph{root} of $z$, respectively. We also define the \emph{weight} of $z$ as 
$$\w(z):=\prod_{i=1}^h\frac 1{\deg(z_i)}.$$
Note that this quantity is the quenched probability that the sequence $(z_0,\ldots,z_h)$ is realized by a \textsc{nbrw} on $G$ starting from $z_0$.  In particular,
\begin{eqnarray}
\label{wbound}
\w(z) & \leq & P^h(z_0,z).
\end{eqnarray}
Our rule for step 1 consists in selecting an half-edge with maximal weight\footnote{For definiteness, let us say that we use the lexicographic order on $\cX$ to break ties.} among all unpaired $z\in\fF$ with height $\h(z)<t/2$ and weight $\w(z)>\wm$, where 
$$\wm:=N^{-\frac 23}.$$ 
The only role of this parameter is to limitate the number of pairs formed during the exploration stage. As outlined in Section \ref{sec:outline}, we shall be interested in 
\begin{eqnarray*}
\fW & := & \sum_{(u,v)\in\cH_x\times\cH_y}\w(u)\w(v){\bf 1}_{\w(u)\w(v)\leq \theta},
\end{eqnarray*}
where $\cH_x$ (resp. $\cH_y$) denotes the set of unpaired half-edges with height $\frac{t}{2}$ and root $x$ (resp. $y$) in $\fF$ at the end of the exploration stage, and where
\begin{eqnarray}
\label{def:theta}
\theta & := & \frac{1}{N(\log N)^2}.
\end{eqnarray}
Write $\Zh$ for the quantity obtained by replacing $\leq$ with $>$ in $\fW$, so that
\begin{eqnarray*}
\fW+\Zh & = & \sum_{(u,v)\in\cH_x\times\cH_y}\w(u)\w(v)\\
& \geq & \sum_{z\in\cH_x\cup\cH_y}\w(z)-1,
\end{eqnarray*}
thanks to the inequality $ab\geq a+b-1$ for $a,b\in[0,1]$. 
Now, let $\fU$ denote the set of unpaired half-edges in $\fF$. By construction,  at the end of the exploration stage, each $z\in\fU$ must have height $t/2$ or weight less than $\wm$, so that 
\begin{eqnarray*}
\fW+\Zh & \geq & \sum_{z\in\fU}\w(z)-\sum_{z\in\fF}\w(z){\bf 1}_{(\w(z)<\wm)}-1.
\end{eqnarray*}
Therefore, Proposition \ref{pr:up} follows from the following four technical lemmas. 
\begin{lemma}
\label{lm:completion}
For every  $\varepsilon>0$,  
\begin{eqnarray*}
\PP\left(P^{t}(x,\pi(y))\leq \frac{\fW-\varepsilon}{N}\right) & = & o\left(\frac 1{N^2}\right)\, .
\end{eqnarray*}
\end{lemma}

 \begin{lemma}\label{lm:small}
For every $\varepsilon>0$, 
\begin{eqnarray*}
\PP\left(\sum_{z\in\fF}\w(z){\bf 1}_{(\w(z)<\wm)}>\varepsilon\right)  & = & o\left(\frac 1{N^2}\right)\, .
\end{eqnarray*}
\end{lemma}

\begin{lemma}\label{lm:large}
For every $\varepsilon>0$, 
\begin{eqnarray*}
\PP\left(\Zh>\Phi(\lambda)+\varepsilon\right)  & = & o\left(\frac 1{N^2}\right)\, .
\end{eqnarray*}
\end{lemma}

\begin{lemma}\label{lm:collisions}
For every $\varepsilon>0$,
\begin{eqnarray*}
\PP\left(x\in\cR,y\in\cR\setminus\cB_x,\sum_{z\in\fU}\w(z)<2-\varepsilon\right) & = & o\left(\frac 1{N^2}\right)\, .
\end{eqnarray*}
\end{lemma}

\section{Proof of Lemma \ref{lm:completion}}
\label{sec:completion}

Combining the representation  (\ref{eq:reversible}) with the observation (\ref{wbound})  yields
\begin{eqnarray*}
P^{t}(x,\pi(y)) & \geq & \sum_{(u,v)\in\cH_x\times\cH_y}\w(u)\w(v){\bf 1}_{\w(u)\w(v)\leq \theta}{\bf 1}_{\pi(u)=v}.
\end{eqnarray*}
The right-hand side can be interpreted as the weight of the uniform pairing chosen at the completion stage, provided we define the weight of a pair $(u,v)$ as
\begin{eqnarray}
\label{eq:weight-pairing}
\w(u)\w(v){\bf 1}_{u\in\cH_x}{\bf 1}_{v\in\cH_y}{\bf 1}_{\w(u)\w(v)\leq \theta}.
\end{eqnarray}
Lemma \ref{lm:completion} now follows from the following general concentration inequality, which we apply conditionally on the exploration stage, with $\cI$ being the set of  half-edges that did not get paired, and weights being given by (\ref{eq:weight-pairing}).
\begin{lemma}\label{prop:conc_ineq_matching}
Let $\cI$ be an even set, $\{w_{i,j}\}_{(i,j)\in \cI\times \cI}$ an array of non-negative weights, and $\pi$ a uniform random pairing on $\cI$. Then for all $a> 0$,
\begin{eqnarray*}
\PP\left(\sum_{i\in\cI} w_{i,\pi(i)}\leq m-a\right) & \leq & \exp\left\{-\frac{a^2}{4 \theta m}\right\},
\end{eqnarray*}
where $m=\frac{1}{|\cI|-1}\sum_{i\in\cI}\sum_{j\neq i}w_{i,j}$ and $\theta=\max_{i\neq j}(w_{i,j}+w_{j,i})$. 
\end{lemma}
Note that in our case, $m  =  \frac{\fW}{|\cI|-1}$. Lemma \ref{lm:completion} follows easily by taking $a=\frac{\varepsilon}{|\cI|-1}$ and observing that $|\cI|-1\leq N$ and $\fW\leq 1$. 

\begin{proof}
We exploit the following concentration result for Stein pairs due to \citet{chatterjee2007stein} (see also \citet[Theorem 7.4]{ross2011fundamentals}): let $Y,Y'$ be bounded variables satisfying
\begin{enumerate}
\item $(Y,Y')\stackrel{d}{=}(Y',Y)$;
\item $\EE[Y'-Y|Y]=-\lambda Y$;
\item $\EE[(Y'-Y)^2|Y]\leq \lambda(bY+c)$,
\end{enumerate}
for some constants $\lambda\in(0,1)$ and $b,c\geq 0$. Then for all $a>0$,
\begin{eqnarray*}
\PP\left(Y\leq -a\right) \leq \exp\left\{-\frac{a^2}{c}\right\}& \textrm{ and } &
\PP\left(Y\geq a\right) \leq  \exp\left\{-\frac{a^2}{ab+c}\right\}.
\end{eqnarray*}
We shall only use the first inequality. Consider the centered variable 
$$Y:=\sum_{i\in\cI} w_{i,\pi(i)}-m,$$ and let $Y'$ be the corresponding quantity for the pairing $\pi'$ obtained from $\pi$ by performing a random switch: two indices $i,j$ are sampled uniformly at random from $\cI$ without replacement, and the pairs $\{i,\pi(i)\}$, $\{j,\pi(j)\}$ are replaced with the pairs $\{i,j\}$, $\{\pi(i),\pi(j)\}$.
This changes the weight by exactly
\begin{equation}
\label{eq:Deltaij}
\Delta_{i,j}:= w_{i,j}+w_{j,i}+w_{\pi(i),\pi(j)}+w_{\pi(j),\pi(i)}-w_{i,\pi(i)}-w_{\pi(i),i}-w_{j,\pi(j)}-w_{\pi(j),j}.
\end{equation}
It is not hard to see that $(\pi,\pi')\stackrel{d}{=}(\pi',\pi)$,  so that (i) holds. 
Moreover, 
\begin{eqnarray*}
\EE[Y'-Y|\pi]& = & \frac{1}{|\cI|(|\cI|-1)}\sum_{i\in\cI}\sum_{j\neq i} \Delta_{i,j}\\
& = & \frac{4}{|\cI|(|\cI|-1)}\sum_{i\in\cI}\sum_{j\neq i}w_{i,j}-\frac{4}{|\cI|}\sum_{i\in\cI}w_{i,\pi(i)}\\
& = & -\frac{4}{|\cI|}Y.
\end{eqnarray*}
Regarding the square $\Delta_{i,j}^2=|\Delta_{i,j}||\Delta_{i,j}|$, we may bound the first copy of $|\Delta_{i,j}|$ by $2\theta$ and the second by changing all minus signs to plus signs in (\ref{eq:Deltaij}), yielding
\begin{eqnarray*}
\EE\left[(Y'-Y)^2|\pi\right]& = & \frac{1}{|\cI|(|\cI|-1)}\sum_{i\in\cI}\sum_{j\neq i} \Delta_{i,j}^2\\
&  \leq & \frac{8\theta}{|\cI|(|\cI|-1)}\sum_{i\in\cI}\sum_{j\neq i}w_{i,j}+\frac{8\theta}{|\cI|}\sum_{i\in\cI}w_{i,\pi(i)}\\
& = & \frac{8\theta}{|\cI|}\left(2m+Y\right).
\end{eqnarray*}
Note that taking conditional expectation with respect to $Y$ does not affect the right-hand side. Thus, (ii) and (iii) hold with $\lambda=\frac{4}{|\cI|}$, $b=2\theta$ and $c=4 m \theta$. 
\end{proof}

\section{Proof of Lemma \ref{lm:small}}
\label{sec:small}
We may fix $z_0\in\{x,y\}$ and restrict our attention to the halved sum
$$Z:=\sum_{z\in\fF}\w(z){\bf 1}_{(\w(z)<\wm)}{\bf 1}_{(z\textrm{ has root }z_0)}.$$
Consider $m=\lfloor\log N\rfloor$ independent \textsc{nbrw}s on $G(\pi)$ starting at $z_0$, each being killed as soon as its weight falls below $\wm$, and write $A$ for the event that their trajectories form a tree of height less than $t/2$.  Clearly,
$\PP\left(A|\pi\right)\geq Z^m.$
Taking expectation and using Markov inequality, we deduce that
\begin{eqnarray*}
\PP\left(Z>\varepsilon\right) & \leq & \frac{\PP\left(A\right)}{\varepsilon^m},
\end{eqnarray*}
where the average is now taken over both the  walks and the pairing. Recalling that $m=\lceil \log N\rceil$, it is more than enough to establish that 
$\PP(A)  =  \left(o(1)\right)^m.$
To do so, we generate the $m$ killed \textsc{nbrw}s one after this other, revealing the underlying pairs along the way, as described in Section \ref{sec:lower-bound}. Given that the first $\ell-1$ walks form a tree of height less than $t/2$, the conditional chance that the $\ell^\textrm{th}$ walk also fulfils the requirement is $o(1)$, uniformly in $1\leq\ell\leq m$. Indeed, 
\begin{itemize}
\item either its weight falls below $\eta=(1/{\log N})^{2}$ before it ever reaches an unpaired half-edge: thanks to the tree structure, there are at most $\ell-1< m$ possible trajectories to follow, so the chance is less than $$m\eta=o(1).$$
\item or the remainder of its trajectory after the first unpaired half-edge has weight less than ${\Delta\wm}/{\eta}$: this part consists of at most $t/2$ half-edges which can be coupled with uniform samples from $\cX$ for a total-variation cost of ${mt^2}/{N}$, as in Section \ref{sec:lower-bound}. Thus, the conditional chance is at most
\begin{eqnarray*}
\qquad\qquad\frac{mt^2}{N}+\PP\left(\prod_{k=1}^{t/2}\deg(X_k^\star) \geq  \frac{\eta}{\Delta\wm}\right)
& = & o(1),
\end{eqnarray*}
by Chebychev's inequality, since 
$\log\left(\frac{\eta}{\Delta\wm}\right) - \frac {\mu t}2 >>  \sigma\sqrt{\frac t2}$.
\end{itemize}

\section{Proof of Lemma \ref{lm:large}}
\label{sec:large}

Set $m=\lceil\log N\rceil$. On $G(\pi)$, let $X^{(1)},\ldots,X^{(m)}$ and $Y^{(1)},\ldots,Y^{(m)}$ be $2m$ independent \textsc{nbrw}s of length $t/2$ starting at $x$ and $y$ respectively. Let $B$ denote the event that their trajectories form two disjoint trees and that for all $1\leq k\leq m$, $$\prod_{\ell=1}^{t/2} \frac{1}{\deg(X^{(k)}_\ell)}\prod_{\ell=1}^{t/2}\frac{1}{\deg(Y^{(k)}_\ell)}>\theta.$$ Then clearly, 
$\PP\left(B|\pi\right)\geq \Zh^m$. Averaging w.r.t. the pairing $\pi$, we see that
\begin{eqnarray*}
\PP\left(\Zh>\Phi(\lambda)+\varepsilon\right) & \leq & \frac{\PP\left(B\right)}{(\Phi(\lambda)+\varepsilon)^m}.
\end{eqnarray*}
Thus, it is enough to establish that 
$\PP(B)  \leq \left(\Phi(\lambda)+o(1)\right)^m$. We do so by generating the $2m$ walks $X^{(1)},Y^{(1)},\ldots,X^{(m)},Y^{(m)}$ one after the other along with the underlying pairing, as above. Given that $X^{(1)},Y^{(1)},\ldots,X^{(\ell-1)},Y^{(\ell-1)}$ already satisfy the desired property, the conditional chance that $X^{(\ell)},Y^{(\ell)}$ also does is at most $\Phi(\lambda)+o(1)$, uniformly in $1\leq \ell \leq m$. Indeed,
\begin{itemize}
\item either one of the two walks attains length $s=\lceil 2\log\log N\rceil$ before reaching an unpaired half-edge: there are at most $\ell-1<m$ possible trajectories to follow for each walk, so the conditional chance is at most 
$$2m2^{-s}=o(1).$$
\item or at least $t-2s$ unpaired half-edges are encountered, and  the product of their degrees falls below  $\frac{1}{\theta}$ with conditional probability at most 
\begin{eqnarray*}
\qquad \frac{4mt^2}N+\PP\left(\prod_{k=1}^{t-2s}\deg(X_k^\star)< \frac 1{\theta}\right)& = & \Phi(\lambda)+o(1),
\end{eqnarray*}
by the same coupling as above and Berry-Essen's inequality (\ref{lb:3}).
\end{itemize}

\section{Proof of Lemma \ref{lm:collisions}}
\label{sec:collisions}

We denote by $\tau$ the (random) number of pairs formed during the exploration stage. 
For $k\geq 0$, we let $\fU_k$ denote the set of unpaired half-edges in the forest after $k\wedge\tau$ pairs have been formed, and we consider the random variable
$$W_k:=\sum_{z\in\fU_k}\w(z).$$
Initially $W_0=2$, and this quantity either stays constant of decreases at each stage, depending on whether the condition appearing in step 3 is satisfied or not. More precisely, denoting by $z_k$ (resp. $z_k'$) the half-edge selected at step 1 (resp. chosen at step 2) of the $k^\textrm{th}$ pair, we have
for all $k\geq 1$, 
\begin{eqnarray*}
W_k & = & W_{k-1}-{\bf 1}_{(k\leq \tau)}\left(\w(z_k){\bf 1}_{(z_k'\in \fU_{k-1}^+)}+\w(z_k'){\bf 1}_{(z_k'\in \fU_{k-1})}\right),
\end{eqnarray*}
where $\fU_{k-1}^+$ is the union of $\fU_{k-1}$ and the set of unpaired neighbours of the roots. Now, let $\{\cG_k\}_{k\geq 0}$ be the natural filtration associated with the exploration stage. Note that $\tau$ is a stopping time, that $\w(z_k)$ is $\cG_{k-1}-$measurable, and that  the conditional law of  $z_k'$ given $\cG_{k-1}$ is uniform on $\cX\setminus\{z_1,\ldots,z_k,z_1'\ldots,z_{k-1}'\}$. Thus,
\begin{eqnarray*}
\EE[W_{k}-W_{k-1}|\cG_{k-1}] & = & -{\bf 1}_{(k\leq\tau)}\frac{\w(z_k)(|\fU_{k-1}^+|-2)+W_{k-1}}{N-2k+1}.\\
\EE\left[(W_{k}-W_{k-1})^2|\cG_{k-1}\right] & = & {\bf 1}_{(k\leq\tau)}\frac{\w(z_k)^2(|\fU_{k-1}^+|-4)+2\w(z_k)W_{k-1}+\sum_{z\in\fU_{k-1}}\w(z)^2}{N-2k+1}.
\end{eqnarray*}
To bound those quantities, observe that each half-edge in $\fU_{k-1}$ has weight at least $\frac{\w(z_k)}{\Delta}$ because its parent has been selected at an earlier iteration and our selection rule ensures that the quantity $\w(z_k)$ is non-increasing with $k$. Consequently, 
\begin{eqnarray*}
 |\fU_{k-1}|\frac{\w(z_k)}{\Delta}\leq \sum_{z\in\fU_{k-1}}\w(z)\leq 2.
\end{eqnarray*}
Combining this with the bound $|\fU_{k-1}^+|\leq |\fU_{k-1}|+2\Delta$, we arrive at 
\begin{eqnarray*}
\EE[W_{k}-W_{k-1}|\cG_{k-1}] & \geq & -{\bf 1}_{(k\leq \tau)}\frac{4\Delta}{N-2k+1}\\
\EE\left[(W_{k}-W_{k-1})^2|\cG_{k-1}\right]
& \leq & {\bf 1}_{(k\leq \tau)}\frac{4\Delta\w(z_k)+2}{N-2k+1}.
\end{eqnarray*}
Now recall that  $\w(z_k)\geq \wm$ and $\h(z_k)<\frac t2$ as per our selection rule, implying 
\begin{eqnarray*}
\wm\tau & \leq & \sum_{k\geq 1}{\w(z_k)}{\bf 1}_{(\tau\geq k)}\ \leq \ 
\sum_{z\in\fF}\w(z){\bf 1}_{\left(\h(z)<\frac t2\right)}\ \leq t.
\end{eqnarray*}
The right-most inequality follows from the fact that the total weight at a given height in $\fF$ is at most $2$ (the total weight being preserved from a parent to its children, if any). 
We conclude that 
\begin{eqnarray*}
\sum_{k=1}^\tau\EE[W_{k}-W_{k-1}|\cG_{k-1}] & \geq & -\frac{4\Delta t}{\wm N-2t}:=-m\\
\sum_{k=1}^\tau\EE\left[(W_{k}-W_{k-1})^2|\cG_{k-1}\right] & \leq & \frac{4\Delta t\wm +2t}{N\wm-2t}:=v.
\end{eqnarray*}
Now, fix $\varepsilon>0$ and consider the martingale $\{M_k\}_{k\geq 0}$ defined by  $M_0=0$ and
$$M_k:=\sum_{i=1}^k\left\{(W_{i-1}-W_i)\wedge\varepsilon-\EE\left[(W_{i-1}-W_{i})\wedge\varepsilon \big|\cG_{i-1}\right]\right\}.$$
Then the increments of $\{M_k\}_{k\geq 0}$ are bounded by $\varepsilon$ by construction, and the above computation guarantees that  almost-surely,
\begin{eqnarray*}
\sum_{k=1}^{\tau}\EE\left[\left(M_k-M_{k-1}\right)^2\big|\cG_{k-1}\right] & \leq & v=N^{-\frac{1}{3}+o(1)}.
\end{eqnarray*}
Thus, the martingale version of Bennett's inequality due to \citet{freedman1975tail} yields
\begin{eqnarray}
\label{freedman}
\PP\left(M_\tau> 7\varepsilon\right)& \leq & \left(\frac{ev}{v+7\varepsilon^2}\right)^{7}=N^{-\frac{7}{3}+o(1)}.
\end{eqnarray}
But on the event $\{x\in\cR,y\in\cR\setminus\cB_x\}$, all paths from the set $\{x,y\}$ to itself must have length at least $h$, and since $h\to\infty$, we must have asymptotically
\begin{eqnarray*}
\{x\in\cR,y\in\cR\setminus\cB_x\}& \subseteq & \left\{\max_{k}(W_{k-1}-W_k)\leq \varepsilon\right\}\\
& \subseteq & \left\{W_0-W_\tau\leq M_\tau+m\right\}
\end{eqnarray*}
With (\ref{freedman}), this proves Lemma \ref{lm:collisions} since $W_0-W_\tau=2-\sum_{z\in\fU}\w(z)$ and $m=o(1)$.

\bibliographystyle{abbrvnat}
\bibliography{biblio}
\end{document}